\documentclass{amsart}

\newtheorem{theorem}{Theorem}[section]
\newtheorem{lemma}[theorem]{Lemma}
\newtheorem{corollary}[theorem]{Corollary}

\theoremstyle{definition}
\newtheorem{definition}[theorem]{Definition}
\newtheorem{example}[theorem]{Example}

\theoremstyle{remark}
\newtheorem{remark}[theorem]{Remark}

\numberwithin{equation}{section}

\newcommand{\R}{\mathbb{R}}
\newcommand{\Z}{\mathbb{Z}}

\newcommand{\T}{\mathbb{T}}

\begin{document}

\title{A General Backwards Calculus of Variations via Duality}


\author[A. B. Malinowska]{Agnieszka B. Malinowska}

\address{Faculty of Computer Science,
Bia{\l}ystok University of Technology,
Bia\l ystok, Poland}

\email{abmalina@pb.bialystok.pl}

\thanks{This work was carried out at the University of Aveiro
via a project of the Polish Ministry of Science
and Higher Education ``Wsparcie miedzynarodowej mobilnosci
naukowcow''.}


\author[D. F. M. Torres]{Delfim F. M. Torres}

\address{Department of Mathematics,
University of Aveiro,
3810-193 Aveiro, Portugal}

\email{delfim@ua.pt}



\date{Submitted to \emph{Optimization Letters} 03-June-2010;
revised 01-July-2010; accepted for publication 08-July-2010.}



\begin{abstract}
We prove Euler-Lagrange and natural boundary
necessary optimality conditions
for problems of the calculus of variations
which are given by a composition of nabla integrals
on an arbitrary time scale.
As an application, we get optimality conditions
for the product and the quotient
of nabla variational functionals.
\end{abstract}

\maketitle


\section{Introduction}

A time scale $\T$ is any nonempty closed subset of $\R$.
The theory of dynamic Euler-Lagrange equations
is a recent field attracting considerable attention
--- see \cite{ZbigDel,b7,b:rui:delfim,bhn:Gus,%
Rui:Del:HO,Mal:Tor:09,Mal:Tor:Wei,MyID:183,Nat}
and references therein. For applications of the
calculus of variations on time scales to economics see
\cite{Almeida:T,Atici:Uysal:08,Basia:Naty:delfim}.
In particular, one obtains the well-known
continuous \cite{P:R:K:10}, discrete \cite{KP},
and quantum calculus of variations \cite{Bang:q-calc}
by choosing $\T=\R$, $\T=\Z$, and
$q^{\mathbb{N}_0}:=\{q^k | k \in \mathbb{N}_0\}$, $q>1$,
respectively.

This paper is dedicated to the study of general (non-classical)
problems of calculus of variations on an arbitrary time scale
$\mathbb{T}$. As a particular case, when $\mathbb{T} =
\mathbb{R}$ one gets the generalized calculus of variations
\cite{CLP} with functionals of the form
\begin{equation*}
H \left(\int_{a}^{b}f(t,x(t),x'(t))dt \right)
\end{equation*}
where $f$ has $n$ components and $H$ has $n$ independent variables.
Problems of calculus of variations of this form appear in practical
applications of economics but cannot be solved using the classical theory
(see \cite{CLP} and the references therein).

In the literature of the calculus of variations on time scales,
the problems are formulated and the results are proved in terms
of the delta or the nabla calculus. Here we use a different
approach. We make use of the duality technique
recently introduced by Caputo \cite{Caputo},
obtaining the results for the nabla variational problems directly from
the results on the delta calculus of variations.
Such duality theory has shown recently to be very useful in control theory
\cite{MyID:177}.

In contrast with \cite{Basia:Composition},
we adopt here a backward perspective, which has proved useful,
and sometimes more natural and preferable,
with respect to applications in economics
\cite{Almeida:T,Atici:2006,Atici:Uysal:08}.
The advantage of the backward approach here promoted becomes
evident when one considers that the
time scales analysis has important
implications for numerical analysts, who often prefer
backward differences rather than forward differences to
handle their computations. This is due to practical implementation reasons
and better stability properties of implicit discretizations
\cite{MR2436490,MyID:177}.

The paper is organized as follows. In Section~\ref{sec:prm} some
preliminaries on the recent duality theory on time scales are presented.
Our results are then given in Section~\ref{sec:Euler}: we
formulate the general (non-classical)
problem of calculus of variations \eqref{vp} on an arbitrary time
scale; we obtain a general formula for the Euler-Lagrange equations
and natural boundary conditions (Theorem~\ref{thm:mr});
and interesting corollaries are presented for the product
(Corollary~\ref{cproduct}) and the quotient (Corollary~\ref{cquotient}).
Finally, in Section~\ref{sec:ex} we illustrate the results
of the paper with two examples.


\section{Preliminaries}
\label{sec:prm}

In this section we review some facts from \cite{Caputo} which we
need for the proof of our main result (Theorem~\ref{thm:mr}).
We begin by briefly recalling the basic definitions,
notations, and facts concerning the delta and nabla differential
calculus on time scales, which can be found in the monographs
\cite{BP,book:ts1,1st:book:ts}.

A {\it time scale} $\mathbb{T}$ is an arbitrary nonempty closed subset
of $\mathbb{R}$. For each time scale $\mathbb{T}$ the following operators are used:
the {\it forward jump operator} $\sigma:\mathbb{T} \rightarrow \mathbb{T}$,
defined by $\sigma(t):=\inf\{s \in \mathbb{T}:s>t\}$ for $t\neq\sup \mathbb{T}$
and $\sigma(\sup\mathbb{T})=\sup\mathbb{T}$ if $\sup\mathbb{T}<+\infty$;
the {\it backward jump operator} $\rho:\mathbb{T} \rightarrow \mathbb{T}$,
defined by $\rho(t):=\sup\{s \in \mathbb{T}:s<t\}$ for $t\neq\inf \mathbb{T}$
and $\rho(\inf\mathbb{T})=\inf\mathbb{T}$ if $\inf\mathbb{T}>-\infty$;
the {\it forward graininess function} $\mu:\mathbb{T} \rightarrow [0,\infty)$,
defined by $\mu(t):=\sigma(t)-t$; and the \emph{backward graininess function}
$\nu:\mathbb{T}\rightarrow[0,\infty)$, defined by $\nu(t)=t - \rho(t)$.
A point $t\in\mathbb{T}$ is called \emph{right-dense},
\emph{right-scattered}, \emph{left-dense} or \emph{left-scattered}
if $\sigma(t)=t$, $\sigma(t)>t$, $\rho(t)=t$,
and $\rho(t)<t$, respectively. We say that $t$ is \emph{isolated}
if $\rho(t)<t<\sigma(t)$, that $t$ is \emph{dense} if $\rho(t)=t=\sigma(t)$.
If $\sup \mathbb{T}$ is finite and left-scattered, we define
$\mathbb{T}^\kappa := \mathbb{T}\setminus \{\sup\mathbb{T}\}$,
otherwise $\mathbb{T}^\kappa :=\mathbb{T}$.
We say that a function
$f:\mathbb{T}\rightarrow\mathbb{R}$ is \emph{delta differentiable}
at $t\in\mathbb{T}^{\kappa}$ if there exists a number
$f^{\Delta}(t)$ such that for all $\varepsilon>0$ there is a
neighborhood $U$ of $t$ such that
$|f(\sigma(t))-f(s)-f^{\Delta}(t)(\sigma(t)-s)|
\leq\varepsilon|\sigma(t)-s|$ for all $s\in U$. We call
$f^{\Delta}(t)$ the \emph{delta derivative} of $f$ at $t$ and $f$ is
said \emph{delta differentiable} on $\mathbb{T}^{\kappa}$ provided
$f^{\Delta}(t)$ exists for all $t\in\mathbb{T}^{\kappa}$.
For $f:\mathbb{T} \rightarrow X$, where $X$ is an arbitrary set,
we define $f^\sigma:=f\circ\sigma$.
A function $f:\mathbb{T}\rightarrow\mathbb{R}$ is called
\emph{rd-continuous} if it is continuous at right-dense points
in $\mathbb{T}$ and its left-sided limits exist (finite)
at left-dense points in $\mathbb{T}$. We denote the
set of all rd-continuous functions by
$C^0_{\textrm{rd}} = C_{\textrm{rd}}
= C_{\textrm{rd}}(\mathbb{T})
= C_{\textrm{rd}}(\mathbb{T}; \mathbb{R})$.
The set of functions $f : \mathbb{T} \rightarrow \mathbb{R}$
that are delta differentiable and whose delta derivative
is rd-continuous is denoted by $C_{\textrm{rd}}^1
= C_{\textrm{rd}}^1(\mathbb{T})
= C_{\textrm{rd}}^1(\mathbb{T}; \mathbb{R})$.

In order to introduce the definition of nabla derivative, one defines
a new set $\mathbb{T}_\kappa$ which is derived from $\mathbb{T}$ as
follows: if  $\mathbb{T}$ has a right-scattered minimum $m$, then
$\mathbb{T}_\kappa=\mathbb{T}\setminus\{m\}$; otherwise,
$\mathbb{T}_\kappa= \mathbb{T}$. In order to simplify expressions,
and similarly as done with composition with $\sigma$, we define
$f^{\rho}(t) := f(\rho(t))$.
We say that a function $f:\mathbb{T}\rightarrow\mathbb{R}$ is
\emph{nabla differentiable} at $t\in\mathbb{T}_\kappa$ if there is a
number $f^{\nabla}(t)$ such that for all $\varepsilon>0$ there
exists a neighborhood $U$ of $t$ such that
$|f^\rho(t)-f(s)-f^{\nabla}(t)(\rho(t)-s)|
\leq\varepsilon|\rho(t)-s|$ for all $s\in U$. We call
$f^{\nabla}(t)$ the \emph{nabla derivative} of $f$ at $t$. Moreover,
we say that $f$ is \emph{nabla differentiable} on $\mathbb{T}$
provided $f^{\nabla}(t)$ exists for all $t \in \mathbb{T}_\kappa$.
Let $\mathbb{T}$ be a time scale,
$f:\mathbb{T}\rightarrow\mathbb{R}$. We say that function $f$ is
\emph{ld-continuous} if it is continuous at left-dense points
and its right-sided limits exist (finite) at all right-dense points.
The set of all ld-continuous functions
$f:\mathbb{T}\rightarrow\mathbb{R}$ is denoted by
$C^0_{\textrm{ld}} = C_{\textrm{ld}}
= C_{\textrm{ld}}(\mathbb{T})
= C_{\textrm{ld}}(\mathbb{T}, \mathbb{R})$,
and the set of all nabla differentiable functions
with ld-continuous derivative by
$C_{\textrm{ld}}^1 = C_{\textrm{ld}}^1(\mathbb{T}, \mathbb{R})$.
For the delta differential/integral calculus
we refer the reader to \cite{BP};
for the nabla differential/integral calculus
we refer the reader to \cite[Chap.~3]{book:ts1}.

Given a time scale $\T$ we define the dual time scale $\T^*$ by
$\T^*:=\{s\in \R|-s\in \T\}$.
If $\rho$ and $\sigma$ denote the backward jump
operator and the forward jump operator associated to $\T$, then we
denote by $\hat{\rho}$ and $\hat{\sigma}$ the jump operators
associated with $\T^*$. If $\nu$ and $\mu$ denote respectively the
backward graininess function and the forward graininess function
associated to $\T$, then we denote by $\hat{\nu}$ and $\hat{\mu}$
the graininess functions associated to $\T^*$.

Given a function $f:\mathbb{T}\rightarrow\mathbb{R}$ we define the
dual function $f^*:\mathbb{T^*}\rightarrow\mathbb{R}$ by
$f^*(s):=f(-s)$ for all $s\in \T^*$.
For a given quintuple $(\T,\rho,\sigma,\nu,\mu)$ its dual will be
$(\T^*,\hat{\rho},\hat{\sigma},\hat{\nu}, \hat{\mu})$ where
$\hat{\rho}$, $\hat{\sigma}$, $\hat{\nu}$, $\hat{\mu}$ are given as follows:
$\hat{\rho}(s)=-\sigma(-s)$, $\hat{\sigma}(s)=-\rho(-s)$,
$\hat{\nu}(s)=\mu^*(s)$, and $\hat{\mu}(s)=\nu^*(s)$.

Let $f:\mathbb{T}\rightarrow\mathbb{R}$. The following holds:
\begin{itemize}
\item[(i)] If $f$ is delta (resp. nabla) differentiable at $t_0\in
\T^\kappa$ (resp. $t_0\in \T_\kappa$), then $f^*$ is nabla (resp.
delta) differentiable at $-t_0\in (\T^*)_\kappa$ (resp. $-t_0\in
(\T^*)^\kappa$) and $f^\Delta(t_0)=-(f^*)^{\hat{\nabla}}(-t_0)$
(resp. $f^\nabla(t_0)=-(f^*)^{\hat{\Delta}}(-t_0)$, where $\Delta$
and $\nabla$ denote the delta and nabla derivatives for the time scale $\T$;
and $\hat{\Delta}$ and $\hat{\nabla}$ denote the delta and nabla derivatives
for the time scale $\T^*$.
\item[(ii)] $f$ belongs to $C_{rd}^1$ (resp. $C_{ld}^1$) if and only if $f^*$ belongs to $C_{ld}^1$ (resp.
$C_{rd}^1$).
\end{itemize}

Let $a,b\in\mathbb{T}$ with $a\leq b$. We define the closed interval
$[a,b]$ in $\mathbb{T}$ by
$[a,b]:=\{t \in \mathbb{T}: a \leq t \leq b\}$.
Along this work we always assume that $[a,b]$ denote an interval
in a given time scale $\mathbb{T}$.

Let $f:[a,b]\rightarrow\mathbb{R}$ be a rd-continuous (resp.
ld-continuous). Then,
$$
\int_a^bf(t)\Delta t
=\int_{-b}^{-a}f^*(s)\hat{\nabla}s
\quad \left(\text{resp. }
\int_a^bf(t)\nabla t=\int_{-b}^{-a}f^*(s)\hat{\Delta}s\right) \, .
$$

\begin{lemma}[\cite{Caputo}]
\label{lemma} For a given Lagrangian $L : [a,b]_\kappa\times
\R\times\R\rightarrow \R$ the following identity holds:
\begin{equation*}
\int_a^b L\left(t,x^{\rho}(t),x^{\nabla}(t)\right)\nabla t
=\int_{-b}^{-a} L^*\left(s,(x^*)^{\hat{\sigma}}(s),(x^*)^{\hat{\Delta}}(s)\right)\hat{\Delta} s
\end{equation*}
for all functions $x\in C^1_{ld}([a,b])$, where the dual Lagrangian
$L^*:[-b,-a]^\kappa \times \R\times\R\rightarrow \R$ is defined by
$L^*(s,x,v)=L(-s,x,-v)$ for all $(s,x,v)\in[-b,-a]^\kappa \times
\R\times\R$.
\end{lemma}

A note on the use of the term \emph{duality} is in order.
Duality theory is a standard concept in optimization theory
that one learns in the earliest of courses beginning with
its applications in linear programming and then in nonlinear programming. This
theory has been extended to many settings, including optimal control theory and
the calculus of variations. Typically, when one considers minimizing a function
$f(\cdot)$ over a set $C$, the dual problem is one of maximizing a related function.
In the present context, the word ``duality'' has a different meaning.
The term duality does not refer here to the classical concept, well known to
researchers in optimization, but to the one recently introduced in
\cite{Caputo} (see also \cite{MyID:177}).


\section{Main Results}
\label{sec:Euler}

Throughout we consider $A,B\in \mathbb{T}$ with $A<B$. Now let
$[a,b]$ with $a,b\in \mathbb{T}$, $b<B$ and $a>A$,
be a subinterval of $[A,B]$.
The general (non-classical) problem of the calculus of variations on
time scales under our consideration consists of extremizing
(\textrm{i.e.}, minimizing or maximizing)
\begin{equation}
\label{vp}
\begin{gathered}
 \mathcal{L}[x]=H\left(\int_{a}^{b}f_{1}(t,x^{\rho}(t),x^{\nabla}(t))\nabla
     t,\ldots, \int_{a}^{b}f_{n}(t,x^{\rho}(t),x^{\nabla}(t))\nabla
     t\right)\\
     (x(a)=x_{a}) \quad (x(b)=x_{b})
 \end{gathered}
\end{equation}
over all $x\in C_{ld}^{1}$. Using parentheses around the end-point
conditions means that these conditions may or may not be present. We
assume that:
\begin{itemize}

\item[(i)] the function $H:\mathbb{R}^{n}\rightarrow \mathbb{R}$ has continuous
partial derivatives with respect to its arguments and we denote them
by $H'_{i}$, $i=1,\ldots,n$;

\item[(ii)] functions $(t,y,v)\rightarrow f_{i}(t,y,v)$ from $[a,b]\times \mathbb{R}^{2}$ to
$\mathbb{R}$, $i=1,\ldots,n$, have partial continuous derivatives
with respect to $y,v$ for all $t\in[a,b]$ and we denote them by
$f_{iy}$, $f_{iv}$;

\item [(iii)] $f_{i}$, $i=1,\ldots,n$, and their partial
derivatives are ld-continuous in $t$ for all $x\in C_{ld}^{1}$.

\end{itemize}

A function $x\in C_{ld}^{1}$ is said to be an admissible function
provided that it satisfies the end-points conditions (if any is
given). The following norm in $C_{ld}^{1}$ is considered:
\begin{equation*}
\|x\|_{1}=\sup_{t\in[a,b]}|x^{\rho}(t)|+\sup_{t\in[a,b]}|x^{\nabla}(t)|.
\end{equation*}

\begin{definition}
An admissible function $\tilde{x}$ is said to be a \emph{weak local
minimizer} (resp. \emph{weak local maximizer}) for \eqref{vp}
if there exists $\delta
>0$ such that $\mathcal{L}[\tilde{x}]\leq \mathcal{L}[x]$ (resp. $\mathcal{L}[\tilde{x}] \geq \mathcal{L}[x]$)
for all admissible $x$ with $\|x-\tilde{x}\|_{1}<\delta$.
\end{definition}

The next theorem gives necessary optimality conditions
for problem \eqref{vp}.

\begin{theorem}
\label{thm:mr}
If $\tilde{x}$ is a weak local solution of problem \eqref{vp},
then the Euler-Lagrange equation
\begin{equation}
\label{Euler}
\sum_{i=1}^{n}H'_{i}(\mathcal{F}_{1}[\tilde{x}],\ldots,
\mathcal{F}_{n}[\tilde{x}])\left(f_{iv}^{\nabla}(t,\tilde{x}^{\rho}(t),\tilde{x}^{\nabla}(t))-
f_{iy}(t,\tilde{x}^{\rho}(t),\tilde{x}^{\nabla}(t))\right)=0
\end{equation}
holds for all $t \in [a,b]_\kappa$, where
$\mathcal{F}_{i}[\tilde{x}]=\int_{a}^{b}f_{i}(t,\tilde{x}^{\rho}(t),\tilde{x}^{\nabla}(t))\nabla
t$, $i=1,\ldots,n$. Moreover, if $x(a)$ is not specified, then
\begin{multline}
\label{nat:l}
\sum_{i=1}^{n}H'_{i}(\mathcal{F}_{1}[\tilde{x}],
\ldots,\mathcal{F}_{n}[\tilde{x}])\Biggl(\Bigl.\int_{a}^{\sigma
(a)}f_{iy}(t,\tilde{x}^{\rho}(t),\tilde{x}^{\nabla}(t))\nabla t\Bigr.\\
-f_{iv}(\sigma(a),\tilde{x}^{\rho}(\sigma
(a)),\tilde{x}^{\nabla}(\sigma(a)))\Biggr)=0 \, ;
\end{multline}
if $x(b)$ is not specified, then
\begin{equation*}
\sum_{i=1}^{n}H'_{i}(\mathcal{F}_{1}[\tilde{x}],\ldots,
\mathcal{F}_{n}[\tilde{x}])f_{iv}(b,\tilde{x}^{\rho}(b),\tilde{x}^{\nabla}(b))=0 \, .
\end{equation*}
\end{theorem}

\begin{proof}
Since $\tilde{x}$ is a weak local extremizer for \eqref{vp}, it follows
by Lemma~\ref{lemma} that $\tilde{x}^*$ is a weak local extremizer for
the dual problem
\begin{equation*}
\begin{gathered}
 \mathcal{L^*}[x^*]=H\left(\int_{-b}^{-a}f_{1}^*(s,(x^*)^{\hat{\sigma}}(s),(x^*)^{\hat{\Delta}}(s)){\hat{\Delta}}
     s,\right.\\
\qquad \qquad \qquad \qquad \qquad \qquad
\ldots, \left.\int_{-b}^{-a}f_{n}^*(s,(x^*)^{\hat{\sigma}}(s),(x^*)^{\hat{\Delta}}(s)){\hat{\Delta}}
     s\right),\\
     (x^*(-a)=x_{a}) \quad (x^*(-b)=x_{b})
 \end{gathered}
\end{equation*}
over all $x^*\in C_{rd}^{1}([-b,-a])$. Applying
\cite[Theorem~3.2]{Basia:Composition}
we conclude that $\tilde{x}^*$
satisfies the following conditions:
\begin{multline*}
\sum_{i=1}^{n}H'_{i}(\mathcal{F}_{1}^*[\tilde{x}^*],\ldots,
\mathcal{F}_{n}^*[\tilde{x}^*])\left((f_{iv}^*)^{\hat{\Delta}}(s,(\tilde{x}^*)^{\hat{\sigma}}(s),
(\tilde{x}^*)^{\hat{\Delta}}(s))\right.\\
-\left.(f_{iy}^*)(s,(\tilde{x}^*)^{\hat{\sigma}}(s),(\tilde{x}^*)^{\hat{\Delta}}(s))\right)=0
\end{multline*}
for all $s \in [-b,-a]^\kappa$, where
$\mathcal{F}_{i}^*[\tilde{x}^*]=\int_{-b}^{-a}f_{i}^*(s,(\tilde{x}^*)^{\hat{\sigma}}(s),(\tilde{x}^*)^{\hat{\Delta}}(s))\hat{\Delta}
s$, $i=1,\ldots,n$. Moreover, if $x^*(-b)$ is not specified, then
\begin{equation}
\label{d:nat:l}
\sum_{i=1}^{n}H'_{i}(\mathcal{F}_{1}^*[\tilde{x}^*],\ldots,\mathcal{F}_{n}^*[\tilde{x}^*])f_{iv}^*(-b,
(\tilde{x}^*)^{\hat{\sigma}}(-b),(\tilde{x}^*)^{\hat{\Delta}}(-b))=0;
\end{equation}
if $x^*(-a)$ is not specified, then
\begin{multline}
\label{d:nat:r}
\sum_{i=1}^{n}H'_{i}(\mathcal{F}_{1}^*[\tilde{x}^*],
\ldots,\mathcal{F}_{n}^*[\tilde{x}^*])\Biggl(f_{iv}^*(\hat{\rho}(-a),(\tilde{x}^*)^{\hat{\sigma}}(\hat{\rho}
(-a)),(\tilde{x}^*)^{\hat{\Delta}}(\hat{\rho}(-a)))\Bigr.\\
+\Bigl.\int_{\hat{\rho}(-a)}^{-a}f_{iy}^*(s,(\tilde{x}^*)^{\hat{\sigma}}(s),(\tilde{x}^*)^{\hat{\Delta}}(s))\hat{\Delta}
s\Biggr)=0.
\end{multline}
By Lemma~\ref{lemma},
\begin{equation*}
\begin{split}
\mathcal{F}_{i}^*[\tilde{x}^*]&=\int_{-b}^{-a}f_{i}^*(s,(\tilde{x}^*)^{\hat{\sigma}}(s),(\tilde{x}^*)^{\hat{\Delta}}(s))\hat{\Delta} s\\
&=\int_{a}^{b}f_{i}(t,\tilde{x}^{\rho}(t),\tilde{x}^{\nabla}(t))\nabla t \\
&=\mathcal{F}_{i}[\tilde{x}],
\end{split}
\end{equation*}
$i=1,\ldots,n$, and
$t\in[a,b]_\kappa$. From the
duality of the Euler-Lagrange equations \cite[Theorem~6.10]{Caputo}
it follows that
\begin{multline*}
(f_{iv}^*)^{\hat{\Delta}}\left(s,(\tilde{x}^*)^{\hat{\sigma}}(s),(\tilde{x}^*)^{\hat{\Delta}}(s)\right)
-f_{iy}^*\left(s,(\tilde{x}^*)^{\hat{\sigma}}(s),(\tilde{x}^*)^{\hat{\Delta}}(s)\right)\\
= f_{iv}^{\nabla}\left(t,\tilde{x}^{\rho}(t),\tilde{x}^{\nabla}(t)\right)
- f_{iy}\left(t,\tilde{x}^{\rho}(t),\tilde{x}^{\nabla}(t)\right),
\end{multline*}
$i=1,\ldots,n$. This establishes relation \eqref{Euler}. Now assume
that $x(b)$ is not specified. By duality,
condition \eqref{d:nat:l} holds. Since
$$f_{iv}^*(-b,(\tilde{x}^*)^{\hat{\sigma}}(-b),(\tilde{x}^*)^{\hat{\Delta}}(-b))
=-f_{iv}(b,\tilde{x}^{\rho}(b),\tilde{x}^{\nabla}(b)),$$
$i=1,\ldots,n$, we have
$$\sum_{i=1}^{n}H'_{i}(\mathcal{F}_{1}[\tilde{x}],\ldots,
\mathcal{F}_{n}[\tilde{x}])f_{iv}(b,\tilde{x}^{\rho}(b),\tilde{x}^{\nabla}(b))=0.$$
If $x(a)$ is not specified, then by duality condition
\eqref{d:nat:r} holds. Observe that
$$f_{iv}^*(\hat{\rho}(-a),(\tilde{x}^*)^{\hat{\sigma}}(\hat{\rho}
(-a)),(\tilde{x}^*)^{\hat{\Delta}}(\hat{\rho}(-a))=-f_{iv}(\sigma(a),\tilde{x}^{\rho}(\sigma
(a)),\tilde{x}^{\nabla}(\sigma(a)),$$ and
$$\int_{\hat{\rho}(-a)}^{-a}f_{iy}^*(s,(\tilde{x}^*)^{\hat{\sigma}}(s),(\tilde{x}^*)^{\hat{\Delta}}(s))\hat{\Delta}
s=\int_a^{\sigma
(a)}f_{iy}\left(t,\tilde{x}^{\rho}(t),\tilde{x}^{\nabla}(t)\right)\nabla t.$$
From the above it follows \eqref{nat:l}.
\end{proof}

Choosing $\mathbb{T}=\mathbb{R}$ in Theorem~\ref{thm:mr} we
immediately obtain the following result:

\begin{corollary}(Th.~3.1 and Eq.~(4.1) in \cite{CLP})
If $\tilde{x}$ is a solution of the problem
\begin{equation*}
\begin{gathered}
\mathcal{L}[x]=H\left(\int_{a}^{b}f_{1}(t,x(t),x'(t))dt,\ldots,
\int_{a}^{b}f_{n}(t,x(t),x'(t))dt\right) \longrightarrow \textrm{extr}\\
(x(a)=x_{a}) \quad (x(b)=x_{b})
\end{gathered}
\end{equation*}
then the Euler-Lagrange equation
\begin{equation*}
\sum_{i=1}^{n}H'_{i}(\mathcal{F}_{1}[\tilde{x}],\ldots,\mathcal{F}_{n}[\tilde{x}])\left(f_{iy}(t,\tilde{x}(t),\tilde{x}'(t))-
\frac{d}{dx}f_{iv}(t,\tilde{x}(t),\tilde{x}'(t))\right)=0
\end{equation*}
holds for all $t \in [a,b]$, where
$\mathcal{F}_{i}[\tilde{x}]=\int_{a}^{b}f_{i}(t,\tilde{x}(t),\tilde{x}'(t))dt$,
$i=1,\ldots,n$. Moreover, if $x(a)$ is not specified, then
\begin{equation*}
\sum_{i=1}^{n}H'_{i}(\mathcal{F}_{1}[\tilde{x}],\ldots,\mathcal{F}_{n}[\tilde{x}])f_{iv}(a,\tilde{x}(a),\tilde{x}'(a))=0;
\end{equation*}
if $x(b)$ is not specified, then
\begin{equation*}
\sum_{i=1}^{n}H'_{i}(\mathcal{F}_{1}[\tilde{x}],\ldots,\mathcal{F}_{n}[\tilde{x}])f_{iv}(b,\tilde{x}(b),\tilde{x}'(b))=0.
\end{equation*}
\end{corollary}

\begin{corollary}\label{cproduct}
If $\tilde{x}$ is a solution of the problem
\begin{equation*}
\begin{gathered}
\mathcal{L}[x]=\left(\int_{a}^{b}f_{1}(t,x^{\rho}(t),x^{\nabla}(t))\nabla
     t\right)\left( \int_{a}^{b}f_{2}(t,x^{\rho}(t),x^{\nabla}(t))\nabla
     t\right) \longrightarrow \textrm{extr}\\
     (x(a)=x_{a}) \quad (x(b)=x_{b})
\end{gathered}
\end{equation*}
then the Euler-Lagrange equation
\begin{multline*}
\mathcal{F}_{2}[\tilde{x}]\left(f_{1v}^{\nabla}(t,\tilde{x}^{\rho}(t),\tilde{x}^{\nabla}(t))-
f_{1y}(t,\tilde{x}^{\rho}(t),\tilde{x}^{\nabla}(t))\right)\\
+\mathcal{F}_{1}[\tilde{x}]\left(f_{2v}^{\nabla}(t,\tilde{x}^{\rho}(t),\tilde{x}^{\nabla}(t))-
f_{2y}(t,\tilde{x}^{\rho}(t),\tilde{x}^{\nabla}(t))\right)=0
\end{multline*}
holds for all $t \in [a,b]_\kappa$. Moreover, if $x(a)$ is not
specified, then
\begin{multline*}
\mathcal{F}_{2}[\tilde{x}]\left(\int_{a}^{\sigma
(a)}f_{1y}(t,\tilde{x}^{\rho}(t),\tilde{x}^{\nabla}(t))\nabla t
-f_{1v}(\sigma(a),\tilde{x}^{\rho}(\sigma
(a)),\tilde{x}^{\nabla}(\sigma(a)))\right)\\
+\mathcal{F}_{1}[\tilde{x}]\left(\int_{a}^{\sigma
(a)}f_{2y}(t,\tilde{x}^{\rho}(t),\tilde{x}^{\nabla}(t))\nabla t
-f_{2v}(\sigma(a),\tilde{x}^{\rho}(\sigma
(a)),\tilde{x}^{\nabla}(\sigma(a)))\right)=0;
\end{multline*}
if $x(b)$ is not specified, then
\begin{equation*}
\mathcal{F}_{2}[\tilde{x}]f_{1v}(b,\tilde{x}^{\rho}(b),\tilde{x}^{\nabla}(b))
+\mathcal{F}_{1}[\tilde{x}]f_{2v}(b,\tilde{x}^{\rho}(b),\tilde{x}^{\nabla}(b))=0.
\end{equation*}
\end{corollary}

\begin{remark}
In the particular case $\mathbb{T}=\mathbb{R}$,
Corollary~\ref{cproduct} gives a result of \cite{CLP}: the
Euler-Lagrange equation associated with the product functional
\begin{equation*}
\mathcal{L}[x]=\left(\int_{a}^{b}f_{1}(t,x(t),x'(t))dt\right)\left(
\int_{a}^{b}f_{2}(t,x(t),x'(t))dt\right)
\end{equation*}
is
\begin{multline*}
\mathcal{F}_{2}[x]\left(f_{1y}(t,x(t),x'(t))-
\frac{d}{dt}f_{1v}(t,x(t),x'(t))\right)\\
+\mathcal{F}_{1}[x]\left(f_{2y}(t,x(t),x'(t))-
\frac{d}{dt}f_{2v}(t,x(t),x'(t))\right)=0
\end{multline*}
and the natural condition at $t=a$, when $x(a)$ is free, becomes
\begin{equation*}
\mathcal{F}_{2}[x]f_{1v}(a,x(a),x'(a))+\mathcal{F}_{1}[x]f_{2v}(a,x(a),x'(a))=0.
\end{equation*}
\end{remark}

\begin{corollary}\label{cquotient}
If $\tilde{x}$ is a solution of the problem
\begin{equation*}
\begin{gathered}
\mathcal{L}[x]=\frac{\int_{a}^{b}f_{1}(t,x^{\rho}(t),x^{\nabla}(t))\nabla
     t}{\int_{a}^{b}f_{2}(t,x^{\rho}(t),x^{\nabla}(t))\nabla t} \longrightarrow \textrm{extr}\\
     (x(a)=x_{a}) \quad (x(b)=x_{b})
 \end{gathered}
\end{equation*}
then the Euler-Lagrange equation
\begin{multline*}
f_{1v}^{\nabla}(t,\tilde{x}^{\rho}(t),\tilde{x}^{\nabla}(t))-
f_{1y}(t,\tilde{x}^{\rho}(t),\tilde{x}^{\nabla}(t))\\
-Q\left(f_{2v}^{\nabla}(t,\tilde{x}^{\rho}(t),\tilde{x}^{\nabla}(t))-
f_{2y}(t,\tilde{x}^{\rho}(t),\tilde{x}^{\nabla}(t))\right)=0
\end{multline*}
holds for all $t \in [a,b]_\kappa$, where
$Q=\frac{\mathcal{F}_{1}[\tilde{x}]}{\mathcal{F}_{2}[\tilde{x}]}$.
Moreover, if $x(a)$ is not specified, then
\begin{multline*}
\int_{a}^{\sigma
(a)}f_{1y}(t,\tilde{x}^{\rho}(t),\tilde{x}^{\nabla}(t))\nabla t
-f_{1v}\left(\sigma(a),\tilde{x}^{\rho}(\sigma
(a)),\tilde{x}^{\nabla}(\sigma(a))\right)\\
-Q\left(\int_{a}^{\sigma
(a)}f_{2y}(t,\tilde{x}^{\rho}(t),\tilde{x}^{\nabla}(t))\nabla t
-f_{2v}\left(\sigma(a),\tilde{x}^{\rho}(\sigma
(a)),\tilde{x}^{\nabla}(\sigma(a))\right)\right)=0;
\end{multline*}
if $x(b)$ is not specified, then
\begin{equation*}
f_{1v}(b,\tilde{x}^{\rho}(b),\tilde{x}^{\nabla}(b))-Qf_{2v}(b,\tilde{x}^{\rho}(b),\tilde{x}^{\nabla}(b))=0.
\end{equation*}
\end{corollary}

\begin{remark}
In the particular situation $\mathbb{T}=\mathbb{R}$,
Corollary~\ref{cquotient} gives the following result of \cite{CLP}:
the Euler-Lagrange equation associated with the quotient functional
\begin{equation*}
\mathcal{L}[x]=\frac{\int_{a}^{b}f_{1}(t,x(t),x'(t))dt}{\int_{a}^{b}f_{2}(t,x(t),x'(t))dt}
\end{equation*}
is
\begin{multline*}
f_{1y}(t,x(t),x'(t))-
Qf_{2y}(t,x(t),x'(t))\\
-\frac{d}{dt}\left[(f_{1v}(t,x(t),x'(t))-
Qf_{2v}(t,x(t),x'(t))\right]=0
\end{multline*}
and the natural condition at $t=a$, when $x(a)$ is free, becomes
\begin{equation*}
f_{1v}(a,x(a),x'(a))- Qf_{2v}(a,x(a),x'(a))=0.
\end{equation*}
\end{remark}


\section{Examples}
\label{sec:ex}

 \begin{example}
Consider the problem
\begin{equation}\label{ex:product}
\begin{gathered}
\mathcal{L}[x]=\left(\int_{-1}^{0}(x^{\nabla}(t))^2\nabla
     t\right)\left(\int_{-1}^{0}tx^{\nabla}(t) \nabla
     t\right)\longrightarrow \min\\
     x(-1)=1, \quad x(0)=0.
 \end{gathered}
\end{equation}
If $\tilde{x}$ is a local minimum of \eqref{ex:product}, then the
Euler-Lagrange equation must hold, i.e.,
\begin{equation}\label{ex:product:euler}
2\tilde{x}^{\nabla\nabla}(t)Q_{2}+Q_{1}=0,
\end{equation}
where
\begin{equation*}
Q_{1}=\mathcal{F}_{1}[\tilde{x}]=\int_{-1}^{0}(\tilde{x}^{\nabla}(t))^2\nabla
t, \quad
Q_{2}=\mathcal{F}_{2}[\tilde{x}]=\int_{-1}^{0}t\tilde{x}^{\nabla}(t)
\nabla t.
\end{equation*}
If  $Q_{2}= 0$, then also $Q_{1}=0$. This contradicts the fact that
on any time scale a global minimizer for the problem
\begin{equation*}
\begin{gathered}
\mathcal{F}_{1}[x]=\int_{-1}^{0}(x^{\nabla}(t))^2\nabla t \longrightarrow \min\\
     x(-1)=1, \quad x(0)=0
\end{gathered}
\end{equation*}
is $\bar{x}(t)=-t$ and $\mathcal{F}_{1}[\bar{x}]=1$. Hence,
$Q_{2}\neq 0$ and equation \eqref{ex:product:euler} implies that
the extremals for problem \eqref{ex:product} are those
satisfying the delta differential equation
\begin{equation}\label{euler}
\tilde{x}^{\nabla\nabla}(t)=-\frac{Q_1}{2Q_2}
\end{equation}
subject to boundary conditions $x(-1)=1$ and $x(0)=0$. A solution of
\eqref{euler} depends on the time scale. Let us solve, for example,
this equation on $\mathbb{T}=\mathbb{R}$ and on
$\mathbb{T}=\left\{-1,-\frac{1}{2},0\right\}$. On
$\mathbb{T}=\mathbb{R}$ we obtain
\begin{equation}\label{sol:P}
x(t)=-\frac{Q_1}{4Q_2}t^2-\frac{4Q_2+Q_1}{4Q_2}t.
\end{equation}
Substituting \eqref{sol:P} into functionals $\mathcal{F}_1$ and
$\mathcal{F}_2$ gives
\begin{equation}\label{equation:Q1,Q2}
\begin{cases}
\frac{48Q_2^2+Q_1^2}{48Q_2^2}=Q_1\\
\frac{12Q_2-Q_1}{24Q_2}=Q_2.
\end{cases}
\end{equation}
Solving the system of equations \eqref{equation:Q1,Q2} we obtain
\begin{gather*}
\begin{cases}
Q_1=0\\
Q_2=0,
\end{cases}
\quad
\begin{cases}
Q_1=\frac{4}{3}\\
Q_2=\frac{1}{3}.
\end{cases}
\end{gather*}
Therefore,
\begin{equation*}
\tilde{x}(t)=-t^2-2t
\end{equation*}
is an extremal for problem \eqref{ex:product} on
$\mathbb{T}=\mathbb{R}$.

The solution of \eqref{euler} on
$\mathbb{T}=\left\{-1,-\frac{1}{2},0\right\}$ is
\begin{gather}\label{euler:T}
x(t)=
\begin{cases}
1 & \text{ if } t=-1\\
\frac{1}{2}+\frac{Q_1}{16Q_2} & \text{ if } t=-\frac{1}{2}\\
0 & \text{ if } t=0 .
\end{cases}
\end{gather}
Constants $Q_1$ and $Q_2$ are determined by substituting
\eqref{euler:T} into functionals $\mathcal{F}_1$ and
$\mathcal{F}_2$. The resulting
 system of equations is
\begin{equation}\label{equation:T}
    \begin{cases}
1+\frac{Q_1^2}{64Q_2^2}=Q_1\\
\frac{1}{4}-\frac{Q_1}{32Q_2}=Q_2.
\end{cases}
\end{equation}
Since system of equations \eqref{equation:T} has no real solutions,
we conclude that there exists no extremizer for problem
\eqref{ex:product} on $\mathbb{T}=\left\{-1,-\frac{1}{2},0\right\}$
among the set of functions that we consider to be admissible.
\end{example}

\begin{example}\label{ex:q:1}
Consider now the problem
\begin{equation}
\label{ex:quotient:1}
\begin{gathered}
\mathcal{L}[x]=\frac{\int_{-2}^{0}(x^{\nabla}(t))^2 \nabla
     t}{\int_{-2}^{0}(x^{\nabla}(t)+(x^{\nabla}(t))^2)\nabla
     t} \longrightarrow \min \\
     x(-2)=4, \quad x(0)=0.
\end{gathered}
\end{equation}
If $\tilde{x}$ is a local minimizer for \eqref{ex:quotient:1}, then
the Euler-Lagrange equation must hold, i.e.,
\begin{equation*}
0=[2\tilde{x}^{\nabla}(t)-Q(1+2\tilde{x}^{\nabla}(t))]^\nabla, \quad
t\in[-2,0]_{\kappa},
\end{equation*}
where
\begin{equation*}
Q=\frac{\int_{-2}^{0}(\tilde{x}^{\nabla}(t))^2 \nabla
     t}{\int_{-2}^{0}(\tilde{x}^{\nabla}(t)+(\tilde{x}^{\nabla}(t))^2)\nabla
     t}.
\end{equation*}
Therefore,
\begin{equation*}
0=2\tilde{x}^{\nabla\nabla}(t)-Q2\tilde{x}^{\nabla\nabla}(t), \quad
t\in[-2,0]{\kappa}.
\end{equation*}
As $x(-2)=4$ and $x(0)=0$, we have  $Q\neq 1$. Thus
$\tilde{x}^{\nabla\nabla}(t)=0$. The solution of the delta
differential equation
$x^{\nabla\nabla}(t)=0$, $x(-2)=4$, $x(0)=0$,
does not depend on the time scale: $\tilde{x}(t)=-2t$
is an extremal for problem \eqref{ex:quotient:1}.
\end{example}


\section*{Acknowledgments}

This work was partially supported by the
\emph{Portuguese Foundation for Science and Technology} (FCT)
through the \emph{Center for Research and Development
in Mathematics and Applications} (CIDMA)
of the University of Aveiro,
cofinanced by the European Community Fund FEDER/POCI 2010.



\end{document}